\documentclass[12pt]{amsart}
\usepackage{amssymb,amsmath}
\usepackage{enumerate}
\usepackage{mathrsfs}
\newtheorem{theorem}{Theorem}[section]

\newtheorem{lemma}[theorem]{Lemma}

\theoremstyle{definition}
\newtheorem{definition}[theorem]{Definition}

\newtheorem{remark}[theorem]{Remark}

\usepackage[utf8]{inputenc}
\newcommand\hlight[1]{\tikz[overlay, remember picture,baseline=-\the\dimexpr\fontdimen22\textfont2\relax]\node[rectangle,fill=white!50,rounded corners,fill opacity = 0.2,draw,thick,text opacity =1] {$#1$};}
\usepackage[bookmarksnumbered, colorlinks, plainpages]{hyperref}

\title[ ]{Norms of basic operators in vector valued model spaces and de Branges spaces}

\author[Dhara]{Kousik Dhara}
\address{Kousik Dhara, Department of Mathematics, Weizmann Institute of Science, Rehovot 7610001, Israel}
\email{kousik.dhara@weizmann.ac.il}

\author[Dym]{Harry Dym}
\address{Harry Dym, Department of Mathematics, Weizmann Institute of Science, Rehovot 7610001, Israel}
\email{harry.dym@weizmann.ac.il}


\numberwithin{equation}{section}
\begin{document}

\begin{abstract}
	Let $\Omega_+$ be either the open  unit disc or the open upper half plane or the open right half plane. In this paper, we compute the norm of the basic operator $A_\alpha=\Pi_\Theta T_{b_\alpha}|_{\mathcal{H}(\Theta)}$ in the vector valued model space $\mathcal{H}(\Theta)=H^m_2 \ominus \Theta H^m_2$ associated with an $m\times m$ matrix valued inner function $\Theta$ in $\Omega_+$ and show that the norm  is attained. Here $\Pi_\Theta$ denotes the orthogonal projection from the Lebesgue space $L^m_2$ onto $\mathcal{H}(\Theta)$ and $T_{b_\alpha}$ is the operator of multiplication by the elementary Blaschke factor $b_{\alpha}$ of degree one with a zero at a point $\alpha\in \Omega_+$. We show that  if $A_\alpha$ is strictly contractive, then its norm may be  expressed in terms of the singular values of $\Theta(\alpha)$. We then extend this evaluation to the more general setting of   vector valued de Branges spaces.
\end{abstract}

\maketitle
\section{Introduction}

Let $\Omega_+$ stand for any one of the three classical domains: (i) the open  unit disc $\mathbb{D}=\{\lambda\in \mathbb{C}:|\lambda|<1\}$, (ii) the open upper half plane $\mathbb{C}_+=\{\lambda\in \mathbb{C}: \text{ Im } \lambda>0\}$, or (iii) the open right half plane $\mathbb{C}_R=\{\lambda\in \mathbb{C}: \text{ Re } \lambda>0\}$. The \textit{Hardy space} $H^m_2$ denotes the Hilbert space of $m\times 1$ vector-valued  holomorphic functions with entries that belong to the scalar Hardy space $H_2(\Omega_+)$   with inner product
\begin{align}
	\label{eq:stip}
	\left<f,g\right>:=\begin{cases}
		\frac{1}{2\pi}\int_0^{2\pi} g(e^{i\theta})^\ast f(e^{i\theta})\, d\theta  & \quad \text{ if } \Omega_+=\mathbb{D}, \\
		\int_{-\infty}^{\infty} g(x)^\ast f(x)\, dx & \quad \text{ if } \Omega_+=\mathbb{C}_+, \\
		\int_{-\infty}^{\infty} g(iy)^\ast f(iy) \, dy & \quad \text{ if } \Omega_+=\mathbb{C}_R,
	\end{cases}
\end{align}
for $f,g\in H^m_2$. The space $H^m_2$ is identified as a closed subspace of the Lebesgue space $L^m_2$ by identifying each function in $H^m_2$ with its nontangential boundary limit (see e.g., \cite{Hoffman,RR}). An $m\times m$ matrix valued  holomorphic function $\Theta$ in $\Omega_+$ is said to be \textit{inner} if $\Theta(\lambda)^\ast \Theta(\lambda)\preceq I_m$ for all $\lambda \in \Omega_+$ (where  $A\preceq B$ for a pair of $m\times m$ matrices means that $B-A$ is positive semidefinite) and $\Theta(\lambda)^\ast\Theta(\lambda)=I_m$,  a.e. on the boundary of $\Omega_+$ (in the sense of nontangential boundary limits). 

For an $m\times m$ matrix valued inner function $\Theta$,  the corresponding vector valued \textit{model space} is the quotient space
\begin{align*}
	\mathcal{H}(\Theta):= H^m_2 \ominus \Theta H^m_2 \cong H^m_2/\Theta H^m_2.  
\end{align*}
The space $\mathcal{H}(\Theta)$ plays an important role in operator theory and function theory; 
see e.g., \cite{bro71, kuzhel, Nikolski 2, NF} .  
For a fixed point $\alpha \in \Omega_+$, the \textit{elementary Blaschke factor at $\alpha$}  is defined by
\begin{align}\label{Eq: Blaschke}
	b_\alpha(\lambda)=
	\begin{cases}
		(\lambda-\alpha)/(1-\overline{\alpha}\lambda)  & \quad \text{ if } \Omega_+= \mathbb{D} \text{ and } \lambda \in \mathbb{C}\setminus \{1/\overline{\alpha}\}, \\
		(\lambda-\alpha)/(\lambda-\overline{\alpha})  & \quad \text{ if } \Omega_+= \mathbb{C}_+ \text{ and } \lambda \in \mathbb{C}\setminus \{\overline{\alpha}\},\\
		(\lambda-\alpha)/(\lambda+\overline{\alpha})  & \quad \text{ if } \Omega_+= \mathbb{C}_R \text{ and } \lambda \in \mathbb{C}\setminus \{-\overline{\alpha}\}.
	\end{cases}
\end{align}	
Let  $T_{b_\alpha}$ be the operator of multiplication by  $b_\alpha$:
\begin{align*}
	(T_{b_\alpha} f)(\lambda)=b_\alpha(\lambda)f(\lambda) \, \, (f\in H^m_2 \text{ and }\lambda \in \Omega_+).
\end{align*}

\begin{definition}
	Let $\alpha\in \Omega_+$ and $\Theta$ be an $m\times m$ matrix valued inner function in $\Omega_+$. The \textit{basic operator}  $A_\alpha : \mathcal{H}(\Theta)\rightarrow \mathcal{H}(\Theta)$  is defined by the formula 
	\begin{align}\label{defn: formula}
		A_\alpha =\Pi_\Theta T_{b_{\alpha}}|_{\mathcal{H}(\Theta)},
	\end{align}
	where  $\Pi_\Theta$ denotes the orthogonal projection from $L^m_2$ onto $\mathcal{H}(\Theta)$.  
\end{definition}

A bounded linear operator $N$ on a Hilbert space $\mathcal{H}$ ($N\in \mathscr{B}({\mathcal{H}})$ for short) is said to be \textit{norm-attaining} if there exists a unit vector $f_0 \in \mathcal{H}$ such that
\[
\| Nf_0\|_{\mathcal{H}} = \|N\|_{\mathscr{B}(\mathcal{H})}.
\]
In this case, we write $N\in \mathcal{N} \mathcal{A}$. It is easy to see that the compact operators are norm attaining. 
Some additional classes of operators that attain their norm are considered in \cite{Bala,Bottcher,Brown,Car-Nev,Pandey-Paulsen}. 

The case $\alpha=0$ and $\Omega_+=\mathbb{D}$ has special interest, because any contraction $T$ on a separable Hilbert space with ${T^\ast }^n \rightarrow 0$
as $n\rightarrow \infty $ (in the strong operator topology) with the defect indices $(m,m)$ is unitarily equivalent to $A_0=\Pi_\Theta T_{b_0}|_{\mathcal{H}(\Theta)} $.  The choices  $\alpha=i$ (resp., $\alpha=1$) and $\Omega_+=\mathbb{C}_+$ (resp., $\Omega_+=\mathbb{C}_R$) enter in the study of dissipative and accretive operators.

In   \cite[ Proposition 3.1]{Bala} it was shown that  when $\Omega_+=\mathbb{D}$, then $A_0\in \mathcal{N} \mathcal{A}$ and $\Vert A_0\Vert=1$ if and only if there exists a nonzero vector $f\in{\mathcal{H}}(\Theta)$ with $f(0)=0$ for Hilbert space valued Hardy spaces.   The norm of $A_0$ was then evaluated when $m=1$; see 
\cite[ Theorem 3.3]{Bala} and \cite[ Section 7, Corollary 3]{GR}. 

In this paper, we exploit the theory of reproducing kernel Hilbert spaces to evaluate the norm of $A_\alpha$ for arbitrary points $\alpha\in\Omega_+$ and every positive integer $m$ and subsequently extend these results  to the general setting of de Branges 
spaces ${\mathcal{B}}({\mathfrak{E}})$. 
Our first main result is:

\begin{theorem}[see Theorem \ref{thm: main}]\label{thm: main intro}
	Let   $\Theta$ be an $m\times m$ matrix valued inner function in $\Omega_+$, and for each point $\alpha\in \Omega_+$  let 
	$$\mathcal{H}(\Theta)_\alpha=\{f\in \mathcal{H}(\Theta): f(\alpha)=0\}.$$ Then $A_\alpha \in \mathcal{N} \mathcal{A} $ and:
	\begin{enumerate}[\rm(1)]
		\item  $\| A_\alpha\|=1$ if and only if   $\mathcal{H}(\Theta)_\alpha \neq\{0\}$.
		
		\item 
		If $\mathcal{H}(\Theta)_\alpha= \{0\}$ and $s_1\geq s_2\geq \cdots\geq s_m$ are the singular values of $\Theta(\alpha)$, then $$\| A_\alpha\|=\underset{1\leq j\leq m}{\max } \, \{s_j: s_j<1\}.$$
	\end{enumerate}
\end{theorem}

The rest of the paper is organized as follows.
In Section 2, we review some basic definitions and known results from the literature. Section 3 is devoted to preliminary analysis  
for the  three domains of interest. In Section 4, we prove the first main result of this paper. Section 5 deals with some reformulations of the basic operator. In Section 6, we extend the first main theorem to  the more general setting of de Branges spaces ${\mathcal{B}}({\mathfrak{E}})$. Since these results require a more extensive introduction, we postpone further discussion until that section and just note that  Theorem \ref{thm: main intro} is a special case of Theorem \ref{thm:mar1a22}. (In the notation of that section, it corresponds to the case that $E_-$ is an $m\times m$ matrix valued inner function and $E_+(\lambda)=I_m$.)

\section{Notation and Preliminaries}
In this section, we shall recall some basic definitions and results from the literature for future use.
Let $\Omega$ be a non empty open subset of $\mathbb{C}$ and let $\mathbb{C}^{m\times m}$ denote the space of all $m\times m$ complex matrices. A Hilbert space $\mathcal{H}$ of $m\times 1$ vector valued functions defined in $\Omega$ is said to be a \textit{reproducing kernel Hilbert space} (RKHS for short) if for each $\omega \in \Omega$ and $u\in \mathbb{C}^m$ there exists a function $K_\omega(\lambda)\in \mathbb{C}^{m \times m}$ such that the following holds:
\begin{enumerate}[\rm(1)]
	\item The function $K_\omega u: \Omega \rightarrow \mathbb{C}^m$ defined by ($K_\omega u)(\lambda)=K_\omega(\lambda) u$ is in $\mathcal{H}$, and 
	\item $\left<f,K_\omega u\right >_{\mathcal{H}}=u^* f(\omega)$, for every $f\in \mathcal{H}.$
\end{enumerate} 
There is only one function $K_\omega(\lambda)$ that meets these two conditions;  it is called the \textit{reproducing kernel} (RK for short) for $\mathcal{H}$.
The RK $K_\omega(\lambda)$ is positive in the following sense:
\begin{align}\label{Eq: pos kernel}
	\sum_{i,j=1}^{n} u_j^*K_{\omega_i}(\omega_j)u_i =\left\langle \sum_{i=1}^n K_{\omega_i}u_i, \sum_{j=1}^n K_{\omega_j}u_j\right\rangle_{\mathcal{H}}\geq 0
\end{align}
for every choice  of $\omega_1, \cdots, \omega_n \in \Omega$, $u_1,\cdots,u_n\in \mathbb{C}^m$ and $n\in \mathbb{N}$. Thus,   $K_\alpha(\beta)=K_\beta(\alpha)^\ast$ and $K_\omega(\omega)  \succeq 0$  for all $\alpha,\beta,\omega\in \Omega$.

An operator version of a theorem of Aronszjan guarantees that every positive kernel $K_\omega(\lambda)$ (in the sense of Eq \eqref{Eq: pos kernel}) gives exactly one RKHS with RK $K_\omega(\lambda)$ (see e.g.,\cite[Theorem 5.3]{Arov-Dym}).

The decomposition 	
\begin{align}\label{decomp}
	\mathcal{H}=\mathcal{H}_\alpha \oplus \{K_\alpha u: u \in \mathbb{C}^m\},
\end{align}
in which $\alpha\in \Omega$ and 
$$
\mathcal{H}_{\alpha}:=\{ f\in \mathcal{H}: f(\alpha)=0\}
$$ 
plays an important role  in the subsequent analysis.

The Hardy space $H^m_2$ is an RKHS with RK 
\begin{align*}
	K_\omega(\lambda)=\frac{I_m}{\rho_\omega(\lambda)} \, \, (\omega,\lambda \in \Omega), \text{ where }
\end{align*}
\begin{align}\label{Eq: kernel}
	\rho_\omega(\lambda)=\begin{cases}
		1-\lambda \overline{w} & \quad \text{ if } \Omega=\mathbb{D}, \\
		-2\pi i (\lambda-\overline{\omega})& \quad \text{ if } \Omega= \mathbb{C}_+, \\
		2\pi  (\lambda+\overline{\omega}) & \quad \text{ if } \Omega= \mathbb{C}_R.
	\end{cases}
\end{align}

Note that $\Omega_+=\{\omega\in\mathbb{C}:\,\rho_\omega(\omega)>0\}$. Correspondingly, we define  
\[
\begin{split}
	\Omega_-&=\{\omega\in\mathbb{C}:\,\rho_\omega(\omega)<0\}\quad\textrm{and}\\
	\Omega_0&=\{\omega\in\mathbb{C}:\,\rho_\omega(\omega)=0\}.
\end{split}
\]

If $\Theta$ is an $m\times m$ matrix valued inner function, then the subspace $\mathcal{H}(\Theta)=H_2^m\ominus \Theta H_2^m$ is again an RKHS with RK
\begin{align*}
	K_\omega (\lambda)=\dfrac{I_{m}-\Theta(\lambda)\Theta(\omega)^{*}}{\rho_\omega(\lambda)} \quad  ( \lambda, \omega \in \Omega). 
\end{align*}

We shall also use the following notation:\\
$N_\alpha(\lambda)=\rho_\alpha(\lambda)K_\alpha(\lambda)=I_{m}-\Theta(\lambda)\Theta(\alpha)^\ast \quad (\lambda, \alpha \in \Omega)$.\\
$\delta_\alpha(\lambda)=\lambda - \alpha \quad (\lambda,\alpha \in \mathbb{C}). $\\
For $\alpha\in \Omega$, the \textit{generalized backward shift operator}$R_\alpha$  is defined by the formula 
\begin{align*}
	(R_\alpha f)(\lambda)=\begin{cases}
		\dfrac{f(\lambda)-f(\alpha)}{\lambda-\alpha} & \quad \text{ if } \lambda\neq \alpha, \\
		f^{'}(\alpha) & \quad \text{ if } \lambda =\alpha,
	\end{cases}
\end{align*}
for functions $f$ that are holomorphic in $\Omega$.

\section{Preliminary Analysis}
In this section we verify some facts that will be needed for the proof of Theorem \ref{thm: main intro}. 
We begin with the following lemma.

\begin{lemma}\label{lem: inv under backward shift}
	If $\alpha\in \Omega_+$ and $\Theta$ is an $m\times m$ matrix valued inner function in $\Omega_+$, 
	then $\mathcal{H}(\Theta)$ is invariant under $R_\alpha$. 
\end{lemma}
\begin{proof}
	The case $\Omega_+=\mathbb{C}_+$ is treated in detail in \cite[Theorem 5.14]{Arov-Dym}; see also \cite[Section 3.2]{Arov-Dym 3} for  relevant estimates. The other two cases may be verified in much the same way.
\end{proof}

\begin{lemma}\label{lem: adjoint}
	If $\alpha\in \Omega_+$, the adjoint $A_\alpha^\ast$ of the operator $A_\alpha$ defined by the formula \eqref{defn: formula} 
	with respect to the inner product \eqref{eq:stip} is
	\begin{align}\label{eq: adjoint}
		A_\alpha^\ast f &= \Pi_{\Theta} (1/b_{\alpha}) f \nonumber \\
		& = \begin{cases}
			-\overline{\alpha} f + (1-|\alpha|^2)R_\alpha f  & \quad \text{ if } \Omega_+=\mathbb{D}, \\
			f + (\alpha-\overline{\alpha})R_\alpha f  & \quad \text{ if } \Omega_+=\mathbb{C}_+, \\
			f + (\alpha+\overline{\alpha})R_\alpha f  & \quad \text{ if } \Omega_+=\mathbb{C}_R ,
		\end{cases}
	\end{align}
	for $f\in {\mathcal{H}}(\Theta)$.
\end{lemma}

\begin{proof} 
	For all three choices of $\Omega_+$, it is readily checked that 
	$$
	\left< A_{\alpha}f,g \right>=\left< f, (1/b_\alpha) g\right>
	= \left < f, \Pi_\Theta((1/b_\alpha)g)\right>.
	$$
	Therefore, $A_\alpha^*g=\Pi_\Theta((1/b_\alpha)g)$ for all $g\in{\mathcal{H}}(\Theta)$. 	
	
	If $\alpha\in\Omega_+$, $\delta_\alpha(\lambda)=\lambda -\alpha $ and $f\in{\mathcal{H}}(\Theta)$, then:
	\begin{equation}
		\label{eq:mar1a22}
		\Omega_+=\mathbb{D} \implies 
		\frac{\displaystyle f} {\displaystyle b_\alpha} =-\overline{\alpha}f+(1-\vert\alpha\vert^2)\frac{f}{\delta_\alpha}=-\overline{\alpha}f+(1-\vert\alpha\vert^2)R_\alpha f+
		(1-\vert\alpha\vert^2)	\frac{f(\alpha)}{\delta_\alpha};
	\end{equation}
	\begin{equation}
		\label{eq:mar1b22}
		\Omega_+=\mathbb{C}_+ \implies 
		\frac{\displaystyle f} {\displaystyle b_\alpha}=f+(\alpha-\overline{\alpha})\frac{f}{\delta_\alpha}=f+(\alpha-\overline{\alpha})R_\alpha f+(\alpha-\overline{\alpha})\frac{f(\alpha)}{\delta_\alpha}
	\end{equation}
	and
	\begin{equation}
		\label{eq:mar1c22}
		\Omega_+=\mathbb{C}_R \implies
		\frac{\displaystyle f} {\displaystyle b_\alpha}=f+(\alpha+\overline{\alpha})\frac{f}{\delta_\alpha}=f+(\alpha+\overline{\alpha})R_\alpha f+(\alpha+\overline{\alpha})\frac{f(\alpha)}{\delta_\alpha}.
	\end{equation}
	Therefore, the explicit formulas in \eqref{eq: adjoint} hold for  all three choices of  $\Omega_+$,  
	since $R_\alpha f\in{\mathcal{H}}(\Theta)$ and $f(\alpha)/\delta_\alpha$ is orthogonal to $H_2^m(\Omega_+)$ in all three cases.
\end{proof}

\begin{remark}
	The adjoint of the operator $A_\alpha$ plays a significant role in the characterization of de Branges spaces; see e.g., Theorem 23 in \cite{de Branges} for $\Omega_+=\mathbb{C}_+$. Also see Theorem 7.1 of \cite{Dym-Sarkar} for $\Omega_+=\mathbb{C}_+$ and Theorem 5.2 of \cite{Dym}  for $\Omega_+=\mathbb{D}$ for spaces of vector valued functions.
\end{remark}

\begin{theorem} \label{thm: adjoint norm}
	If $\alpha \in \Omega_+$ and $\Theta $ is inner in $\Omega_+$, then
	\begin{align}\label{eq: adjoint norm}
		\|A_\alpha ^\ast f\|^2=\|f\|^2-\rho_\alpha(\alpha)f(\alpha)^\ast f(\alpha) \text{ for all } f\in \mathcal{H}(\Theta).
	\end{align}
\end{theorem}

\begin{proof} 
	We first observe that 
	\begin{align*}
		\|f\|^2- \|A_\alpha ^\ast f\|^2 &=\Vert (1/b_\alpha)f\Vert^2-\Vert \Pi_\Theta (1/b_\alpha)f\Vert^2	\\
		&= \left< (I-\Pi_\Theta)(1/b_\alpha)f, (1/b_\alpha)f \right> \\
		&= \left< (I-\Pi_\Theta)(1/b_\alpha)f, (I-\Pi_\Theta)(1/b_\alpha)f \right> \\
		&=\Vert (I-\Pi_\Theta)(1/b_\alpha) f\Vert^2.
	\end{align*}   
	But, in view of formulas \eqref{eq:mar1a22}--\eqref{eq:mar1c22}, 
	\begin{equation}
		\label{eq:mar1g22}
		(I-\Pi_\Theta)\frac{f}{b_\alpha}=\left\{\begin{array}{ll}(1-\vert\alpha\vert^2) f(\alpha)\delta_\alpha^{-1}&\quad\textrm{if}\ 
			\Omega_+=\mathbb{D},\\ 
			(\alpha-\overline{\alpha})	f(\alpha)\delta_\alpha^{-1}&\quad\textrm{if}\ 
			\Omega_+=\mathbb{C}_+,\\
			(\alpha+\overline{\alpha})f(\alpha)\delta_\alpha^{-1}&\quad\textrm{if}\ 
			\Omega_+=\mathbb{C}_R.
		\end{array}
		\right.
	\end{equation}
	Thus, to complete the verification of formula \eqref{eq: adjoint norm}, it remains only to compute 
	$\Vert f(\alpha)/\delta_\alpha\Vert^2$ 
	in the norm based on the appropriate inner product \eqref{eq:stip} for $\Omega_+$.

	\begin{equation}
		\label{eq:mar1d22}
		\Omega_+=\mathbb{D} \implies 
		\Vert f(\alpha)/\delta_\alpha\Vert^2=f(\alpha)^*f(\alpha)\,\frac{1}{2\pi}\int_0^{2\pi} \vert e^{i\theta}-\alpha\vert^{-2}d\theta=
		\frac{f(\alpha)^*f(\alpha)}{1-\vert \alpha\vert^2}
	\end{equation}
	\begin{equation}
		\label{eq:mar1e22}
		\Omega_+=\mathbb{C}_+ \implies 
		\Vert f(\alpha)/\delta_\alpha\Vert^2=f(\alpha)^*f(\alpha)\,\int_{-\infty}^\infty \vert\mu-\alpha\vert^{-2}d\mu=
		2\pi i \frac{f(\alpha)^*f(\alpha)}{\alpha-\overline{\alpha}}
	\end{equation}
	\begin{equation}
		\label{eq:mar1e22}
		\Omega_+=\mathbb{C}_R \implies 
		\Vert f(\alpha)/\delta_\alpha\Vert^2=f(\alpha)^*f(\alpha)\,\int_{-\infty}^\infty \vert i\nu-\alpha\vert^{-2}d\nu=
		2\pi  \frac{f(\alpha)^*f(\alpha)}{\alpha+\overline{\alpha}}.
	\end{equation}
	Formula \eqref{eq: adjoint norm} drops out by combining formulas. \end{proof}

\section{The first main result}
In this section, we prove the first main result of this paper. We begin with the following lemma.

\begin{lemma}\label{lem: main lemma}
	
	If $\alpha \in \Omega_+$, $u\in\mathbb{C}^m$ and  $f=K_\alpha u\ne 0$, then:
	\begin{enumerate}[\rm(1)]
		\item	$\|A_\alpha  ^\ast f\|^2=u^\ast K_\alpha(\alpha)u-\rho_\alpha(\alpha) u^\ast K_\alpha(\alpha)^2u$.
		\vspace{0.5cm}
		\item 	$\dfrac{\|A_\alpha  ^\ast f\|^2}{\|f\|^2}= \dfrac{w^* \Theta(\alpha)\Theta(\alpha)^*w}{w^*w}$, where $w= N_\alpha(\alpha)^{1/2}u$ and  $N_\alpha(\alpha)=I_m-\Theta(\alpha)\Theta(\alpha)^\ast$.
	\end{enumerate}
\end{lemma}

\begin{proof} 
	In view of Theorem \ref{thm: adjoint norm},   
	\begin{align*}
		\|A_\alpha  ^\ast f\|^2 &=\|K_\alpha u\|^2-\rho_\alpha(\alpha)u^\ast  K_\alpha(\alpha)K_\alpha(\alpha) u \\ 
		&= u^*K_\alpha (\alpha)u -\rho_\alpha(\alpha)u^\ast K_\alpha(\alpha)^2 u,
	\end{align*}
	which yields $(1)$. 
	
	Thus, as  $K_\alpha(\alpha)=N_\alpha(\alpha)/\rho_\alpha(\alpha)$,  
	\begin{align*}
		\dfrac{\|A_\alpha ^\ast f\|^2}{\|f\|^2} &= \dfrac{\|K_\alpha u\|^2-\rho_\alpha(\alpha)u^* K_\alpha(\alpha)^2 u}{u^* K_\alpha(\alpha)u}\\
		&= \dfrac{u^* N_\alpha (\alpha) u-u^* N_\alpha(\alpha)^2 u}{u^*N_\alpha(\alpha) u} \\
		&= \dfrac{u^* N_\alpha (\alpha)^{1/2}(I_m- N_\alpha(\alpha)) N_\alpha(\alpha)^{1/2} u}{u^*N_\alpha(\alpha) u} \\
		&= \frac{w^* \Theta(\alpha)\Theta(\alpha)^* w}{w^* w}.
	\end{align*}
	Therefore, (2) holds. 	
\end{proof}

The next theorem relies on the singular value decomposition of the matrix $\Theta(\alpha)$.
\begin{theorem}\label{thm: NA on subspace}
	Let $\alpha\in \Omega_+$ and  $\Theta$ be an $m\times m$ matrix valued  inner function in $\Omega_+$. Let $\mathcal{H}=\mathcal{H}(\Theta)$ and 
	suppose that $\Theta(\alpha)$ has the singular values $s_1\geq s_2\geq \cdots\geq s_m$ and that 
	$s_1=\cdots=s_k=1$ and $s_{k+1}<1$ for some $k\in \{1, \cdots, m\}.$ Let
	\begin{align*}
		\mathcal{M}_\alpha=\{ K_\alpha u : u \in \mathbb{C}^m \}.
	\end{align*} 
	Then 
	\begin{equation}
		\label{eq:mar2a22}
		\|A_\alpha ^\ast  |_{\mathcal{M}_\alpha} \| =s_{k+1}
	\end{equation}
	and the norm is attained.

\end{theorem}

\begin{proof}
	Let $f$ be a non-zero vector in $\mathcal{M}_\alpha$. Then $f=K_\alpha u$ for some $u\in \mathbb{C}^m$. In view of  Lemma \ref{lem: main lemma}(2), 
	\begin{align}\label{Eq: important}
		\frac{\|A_\alpha  ^\ast f\|^2}{\|f\|^2} 	
		&=\frac{w^* \Theta(\alpha)\Theta(\alpha)^*w}{w^*w}.
	\end{align}
	Suppose $\Theta (\alpha) = V S 
	U^*$ is the singular value decomposition of $\Theta(\alpha)$,	
	where $S$ is the diagonal matrix consisting of the singular values $s_1\geq s_2\geq  \cdots\geq s_m$ of $\Theta (\alpha)$ and the  matrices $U$ and $V$ are unitary. Let   $V=\begin{bmatrix}
		v_1 & v_2 & \cdots & v_m
	\end{bmatrix}$. 
	Then
	\begin{align*}
		N_\alpha(\alpha) &=I_m-\Theta(\alpha)\Theta(\alpha)^\ast\\ &=I_m- VS^2V^* \\
		&= V(I_m-S^2) V^* \\
		&=V\begin{bmatrix}
			1-s_1^2\\
			& \ddots \\
			&&&1-s_m^2\\
		\end{bmatrix} V^*.
	\end{align*}
	Hence, if $s_1=s_2=\cdots=s_k=1$ and $s_{k+1}<1$, then 
	\begin{align*}
		N_\alpha(\alpha)= \begin{bmatrix}
			v_{k+1}  & \cdots & v_m
		\end{bmatrix}
		\begin{bmatrix}
			1-s_{k+1}^2\\
			&\ddots\\
			&&& 1-s_m^2\\
		\end{bmatrix}
		\begin{bmatrix}
			v_{k+1} & \cdots & v_m
		\end{bmatrix}^\ast.
	\end{align*}
	Thus, range $(N_\alpha(\alpha))=\text{span }\{v_{k+1}, \cdots, v_m\}.$
	Then we have by Equation \eqref{Eq: important}
	\begin{align*}
		\frac{\|A_\alpha  ^\ast f\|^2}{\|f\|^2} 
		&= \frac{w^*V S ^{2}V^*w}{w^*w}\\
		&=\frac{ (SV^{*}w)^{*}(SV^*w) }{w^*w}\\
		&= \frac{\|S V^*w\|^2}{\|w\|^2}.
	\end{align*}
	Since $w=N_\alpha(\alpha)^{1/2}u \in \text{ span } \{ v_{k+1}, \cdots, v_m\}$,
	\begin{align*}
		w=\sum_{j=k+1}^{m} c_j v_j \quad\textrm{and}\quad V^*w=\sum_{j=k+1}^{m} c_j e_j 		
	\end{align*}
	for some complex numbers $c_{k+1}, \cdots, c_m$, where $e_j$ denotes the $j$'th column of $I_m$.  
	Hence 
	\begin{align*}
		\|SV^*w\|^2=\sum_{j=k+1}^{m} s_j^2 |c_j|^2 \leq s_{k+1}^2 \sum_{j=k+1}^{m}|c_j|^2= s_{k+1}^2\|w\|^2.
	\end{align*}
	Consequently, 
	\begin{align*}
		\max \frac{\|SV^*w\|^2}{\|w\|^2}= s_{k+1}^2.
	\end{align*}
	where the maximum is achieved by choosing $u=v_{k+1}$. 
	Thus, \eqref{eq:mar2a22} holds and the norm is attained.
\end{proof}
Now we are ready to prove our first main theorem. Before that, we recall a well known result (see e.g., \cite[Proposition 2.5]{Car-Nev}): A bounded linear operator $N\in \mathcal{N} \mathcal{A}$ if and only if $N^\ast \in \mathcal{N} \mathcal{A}$. 	

\begin{theorem}\label{thm: main}
	Let  $\Theta$ be an $m\times m$ matrix valued inner function in $\Omega_+$ and recall that for each point $\alpha\in \Omega_+$,  
	$\mathcal{H}(\Theta)_\alpha=\{f\in \mathcal{H}(\Theta): f(\alpha)=0\}$. Then $A_\alpha \in \mathcal{N} \mathcal{A}$ and:
	\begin{enumerate}[\rm(1)]
		\item  $\| A_\alpha\|=1$ if and only if   $\mathcal{H}(\Theta)_\alpha \neq\{0\}$.
		
		\item If  $\mathcal{H}(\Theta)_\alpha= \{0\}$ and $s_1\geq s_2\geq \cdots\geq s_m$ are the singular values of $\Theta(\alpha)$, then $$\| A_\alpha\|=\underset{1\leq j\leq m}{\max } \, \{s_j: s_j<1\}.$$
	\end{enumerate}
\end{theorem}

\begin{proof} To make the logic of the proof transparent, we divide the argument into three short steps.
	\bigskip
	
	\noindent
	{\bf 1.}\ {\it If ${\mathcal{H}}(\Theta)_\alpha\ne \{0\}$, then $A_\alpha\in\mathcal{N} \mathcal{A}$ and $\Vert A_\alpha\Vert=1$}.
	\bigskip
	
	If $\mathcal{H}(\Theta)_\alpha \neq \{0\}$,  
	then there exists a non-zero vector $f\in \mathcal{H}(\Theta)$ such that $f(\alpha)=0$. Therefore,  Theorem \ref{thm: adjoint norm}, ensures that  
	\begin{align*}
		\|A_\alpha ^\ast f\|=\|f\|,
	\end{align*}
	and hence that $\| A_\alpha ^\ast\|=1$ and $A_\alpha^\ast  \in \mathcal{N} \mathcal{A}$. Thus, $A_\alpha \in \mathcal{N} \mathcal{A}$ and $\|A_\alpha\|=1$. 
	\bigskip
	
	\noindent
	{\bf 2.} \ {\it If ${\mathcal{H}}(\Theta)_\alpha =\{0\}$, then $A_\alpha\in\mathcal{N} \mathcal{A}$, $\Vert A_\alpha\Vert<1$ and (2) holds}.
	\bigskip
	
	If $\mathcal{H}(\Theta)_\alpha = \{0\}$, then
	the decomposition $\mathcal{H}(\Theta)= \mathcal{H}(\Theta)_\alpha \oplus \mathcal{M}_\alpha=\mathcal{M}_\alpha $, where $\mathcal{M}_\alpha=\{K_\alpha u: u\in \mathbb{C}^m\}$.     By Theorem \ref{thm: NA on subspace}, we have $A_\alpha^\ast  \in \mathcal{N} \mathcal{A}$ and
	\begin{align*}
		\| A_{\alpha}^\ast\|=\underset{1\leq j\leq m}{\max } \, \{s_j: s_j<1\},
	\end{align*}
	where $s_1\geq s_2\geq\cdots\geq s_m$ are the singular values of $\Theta(\alpha)$. Thus, $A_\alpha  \in \mathcal{N} \mathcal{A}$,  $\Vert A_\alpha\Vert=\Vert A_\alpha^*\Vert<1$ and (2) holds.
	\bigskip
	
	\noindent
	{\bf 3.} \ {\it $A_\alpha\in\mathcal{N} \mathcal{A}$ and  (1) holds}.
	\bigskip
	
	Steps 1 and 2 cover all possibilities. Therefore, $A_\alpha  \in \mathcal{N} \mathcal{A}$ and (1) holds. Nevertheless, it is reassuring to observe that 
	if $\|A_\alpha \|=1$, then, as  $A_\alpha ^\ast\in \mathcal{N} \mathcal{A}$, there exists a non-zero vector $g\in \mathcal{H}(\Theta)$ such that
	\begin{align*}
		\| A_\alpha ^\ast g\|=\|g\|.
	\end{align*}
	Consequently,  $\|g(\alpha)\|^2=g(\alpha)^*g(\alpha)=0$ by Theorem \ref{thm: adjoint norm}.  Thus,  $\mathcal{H}(\Theta)_\alpha \neq \{0\}$. This gives an independent proof of the converse of the implication in Step 1, and so too another way to complete the proof of (1). 
	
\end{proof}

\begin{remark}\label{rem: Bala}
	The decomposition $\mathcal{H}(\Theta)= \mathcal{H}(\Theta)_\alpha \oplus \mathcal{M}_\alpha $ where $\mathcal{M}_\alpha=\{K_\alpha u: u\in \mathbb{C}^m\}$ implies that
	\begin{align*}
		\mathcal{H}(\Theta)_{\alpha}\neq \{0\} \text{ if and only if } \dim \, \mathcal{H}(\Theta)> \dim \, \mathcal{M}_\alpha.
	\end{align*}

\end{remark}

\section{Equivalent Reformulations}
In this section, we present some equivalent reformulations of the basic operator $A_\alpha$, that will be expressed in terms of the notation \eqref{Eq: Blaschke} for $b_\alpha$, \eqref{Eq: kernel} for $\rho_\alpha$ and the auxiliary notation:
\begin{align*}
	f^\#(\lambda)=\begin{cases}
		f\left(1/ \overline{\lambda} \right)^\ast &\quad \text{ if } \Omega_+=\mathbb{D} \quad \text{ and } \lambda\neq 0, \\
		f(\overline{\lambda})^\ast & \quad \text{ if } \Omega_+=\mathbb{C}_+ , \\
		f(-\overline{\lambda})^\ast & \quad \text{ if } \Omega_+=\mathbb{C}_R,
	\end{cases}
\end{align*}
and for $\alpha\in \Omega_+$
\begin{align*}
	\varphi_\alpha(\lambda)=\begin{cases}
		b_{-\alpha}(\lambda) & \quad \text{ if } \Omega_+=\mathbb{D}, \\
		\lambda	(\alpha-\overline{\alpha})/2i + (\alpha+\overline{\alpha})/2 & \quad \text{ if } \Omega_+=\mathbb{C}_+, \\
		\lambda	(\alpha+\overline{\alpha})/2  + (\alpha-\overline{\alpha})/2 & \quad \text{ if } \Omega_+=\mathbb{C}_R.
	\end{cases}
\end{align*}

\noindent We begin with formulas for $A_\alpha$.
\begin{lemma}\label{lem: basic op expression}
	If $\alpha\in \Omega_+$ and $\Theta$ is an $m\times m$ matrix valued inner function in $\Omega_+$, then
	\begin{align}\label{eq: basic op}
		A_\alpha f= \begin{cases}
			b_{\alpha} f - \dfrac{\rho_\alpha(\alpha)}{\overline{\alpha}}\dfrac{\Theta(\Theta ^\# f)(1/\overline{\alpha})}{\rho_\alpha} & \quad  \text{ if } \alpha\neq 0 \text{ and }\Omega_+=\mathbb{D}, \\
			b_0f - \Theta \lim\limits_{\beta\rightarrow 0} \dfrac{ (\Theta^\# f)(1/\overline{\beta})}{\overline{\beta}}  & \quad  \text{ if } \alpha= 0 \text{ and }\Omega_+=\mathbb{D}, \\
			b_\alpha f +\dfrac {\rho_\alpha(\alpha)}{\rho_\alpha} \Theta(\Theta^\# f)(\overline{\alpha})  & \quad \text{ if } \Omega_+=\mathbb{C}_+,\\
			b_\alpha f +\dfrac{\rho_\alpha(\alpha)}{\rho_\alpha} \Theta(\Theta^\# f)(-\overline{\alpha})  & \quad  \text{ if } \Omega_+=\mathbb{C}_R,
		\end{cases}
	\end{align}
	for all $f\in {\mathcal{H}}(\Theta)$.
\end{lemma}

\begin{proof}
	We consider the case when $\Omega_+=\mathbb{C}_+$.
	Let $f\in \mathcal{H}(\Theta)$. Then
	\begin{align*}
		b_\alpha(\lambda) f(\lambda)=\dfrac{\lambda-\alpha}{\lambda-\overline{\alpha}}f(\lambda)=f(\lambda)+\dfrac{\overline{\alpha}-\alpha}{\lambda-\overline{\alpha}}f(\lambda)=f(\lambda)+\dfrac{\overline{\alpha}-\alpha}{\delta_{\overline{\alpha}}(\lambda)} f(\lambda),
	\end{align*}
	and hence
	\begin{align*}
		\Pi_\Theta b_\alpha f = f+(\overline{\alpha}-\alpha)\Pi_\Theta \dfrac{f}{\delta_{\overline{\alpha}}}.
	\end{align*}
	The last term is evaluated by observing that for every vector $u\in \mathbb{C}^m$ and every point $\omega\in \mathbb{C}_+$
	\begin{align*}
		u^\ast\left (\Pi_\Theta \dfrac{f}{\delta_{\overline{\alpha}}} \right)(\omega) = \left< \dfrac{f}{\delta_{\overline{\alpha}}}, K_\omega u \right> 
		&= \left< \dfrac{f}{\delta_{\overline{\alpha}}}, \dfrac{(I_m-\Theta \Theta(\omega)^\ast)u}{\rho_\omega} \right> \\
		&= u^\ast \dfrac{f(\omega)}{\omega-\overline{\alpha}} - \left< \dfrac{f}{\delta_{\overline{\alpha}}}, \dfrac{\Theta \Theta(\omega)^\ast u}{\rho_\omega} \right>,
	\end{align*}
	since $\dfrac{f}{\delta_{\overline{\alpha}}}\in H^m_2$ and $\dfrac{I_m}{\rho_\omega}$ is the RK for $H^m_2$.
	Moreover, 
	\begin{align*}
		\left< \dfrac{f}{\delta_{\overline{\alpha}}}, \dfrac{\Theta \Theta(\omega)^\ast u}{\rho_\omega} \right>	&=  \left<  \dfrac{\Theta^\# f}{\delta_{\overline{\alpha}}}, \dfrac{ \Theta(\omega)^\ast u}{\rho_\omega} \right> \\
		&=  \dfrac{u^\ast \Theta(\omega)}{2\pi i} \int_{-\infty}^{\infty} \dfrac{(\Theta^\# f)(\mu)}{(\mu-\overline{\alpha})(\mu-\omega)} \, d\mu \\
		&=- u^\ast \Theta(\omega) \dfrac{(\Theta^\# f)(\overline{\alpha})}{\overline{\alpha}-\omega},	
	\end{align*}
	since $\Theta^\# f/\delta_\omega \in L^m_2 \ominus H^m_2$.
	Thus,
	\begin{align*}
		\Pi_\Theta b_\alpha f =f+ \dfrac{(\overline{\alpha}-\alpha)}{ \delta_{\overline{\alpha}}}f + \left( \dfrac{\alpha-\overline{\alpha}}{\delta_{\overline{\alpha}}} \right)\Theta (\Theta^\# f)(\overline{\alpha}),
	\end{align*} 
	which coincides with the formula in \eqref{lem: basic op expression} for $\Omega_+=\mathbb{C}_+$.  The formulas for $\mathbb{D}$ and $\mathbb{C}_R$ can be verified in much the same way.	
\end{proof}

\begin{theorem}
	Let $\Theta$ be an $m\times m$ matrix valued inner function in $\Omega_+$ and let $\widetilde{\Theta}(\lambda)=\Theta(\varphi_\alpha(\lambda))$ for $\alpha \in \Omega_+$. Then $\widetilde{\Theta}$ is an $m\times m$ matrix valued inner function in $\Omega_+$ and: 
	\begin{enumerate}[\rm(1)]
		\item  $A_\alpha$ is unitarily equivalent to the model operator $A_0$  in $\mathcal{H}(\widetilde{\Theta})$, when $\Omega_+=\mathbb{D}$,
		
		\item  $A_\alpha$ is unitarily equivalent to $A_i$ in $\mathcal{H}(\widetilde{\Theta})$, when $\Omega_+=\mathbb{C}_+$, and 
		
		\item $A_\alpha$ is unitarily equivalent to $A_1$ in $\mathcal{H}(\widetilde{\Theta})$, when $\Omega_+=\mathbb{C}_R$.
		
		\item The operator $A_0$ acting in  $\mathcal{H}(\Theta)=H^m_2(\mathbb{D}) \ominus \Theta H^m_2 (\mathbb{D})$ is unitarily equivalent to the operator $A_i$ acting in $\mathcal{H}(\Theta_0)=H^m_2(\mathbb{C}_+) \ominus \Theta_0 H^m_2 (\mathbb{C}_+)$, where $\Theta_0(\mu)=\Theta\left(\dfrac{\mu-i}{\mu+i}\right)$ for all $\mu\in \mathbb{C}_+$.
		
		\item The operator $A_0$ acting in  $\mathcal{H}(\Theta)=H^m_2(\mathbb{D}) \ominus \Theta H^m_2 (\mathbb{D})$ is unitarily equivalent to the operator $A_1$ acting in $\mathcal{H}(\Theta_1)=H^m_2(\mathbb{C}_R) \ominus \Theta_1 H^m_2 (\mathbb{C}_R)$, where $\Theta_1(\mu)=\Theta\left(\dfrac{\mu-1}{\mu+1}\right)$ for all $\mu\in \mathbb{C}_R$.
	\end{enumerate}
\end{theorem}

\begin{proof} It is readily checked that $\widetilde{\Theta}$  is an $m\times m$ matrix valued inner function in $\Omega_+$. The rest of the proof is broken into steps.
	\bigskip
	
	\noindent	{\bf 1.} \textit{Verification of $(1)$$:$} Suppose $\Omega_+=\mathbb{D}$. Proposition 4.1 in \cite{Cima} can be extended to the matrix case to show that the operator $A_\alpha$ acting in $\mathcal{H}(\Theta)$ is unitarily equivalent to the model operator $A_0$ in $\mathcal{H}(\widetilde{\Theta})$.
	More precisely,  if $V_\alpha: H^m_2(\mathbb{D}) \rightarrow H^m_2(\mathbb{D})$ is defined by the formula
	\begin{align}
		(V_\alpha f)(\lambda)= \dfrac{\sqrt{1-|\alpha|^2}}{1+\lambda \bar{\alpha}} f(\varphi_\alpha(\lambda)) \quad (f\in H^m_2, \lambda \in \mathbb{D}),
	\end{align}
	then $V_\alpha$ is a unitary operator that maps $\mathcal{H}(\Theta)$ onto $\mathcal{H}(\widetilde{\Theta})$ and
	\begin{align*}
		V_\alpha A_\alpha f= \Pi_{\widetilde{\Theta}} \, b_0 V_\alpha f=A_0 V_\alpha f \quad (f\in \mathcal{H}(\Theta)).
	\end{align*} 
	
	\bigskip
	
	\noindent
	{\bf 2.} \textit{Verification of $(2)$$:$} Suppose $\Omega_+=\mathbb{C}_+$.
	Let $V_\alpha: H^m_2(\mathbb{C}_+)\rightarrow H^m_2(\mathbb{C}_+)$ be defined by	the formula
	\begin{align}
		(V_\alpha f)(\lambda)=\sqrt{(\alpha-\overline{\alpha})/2i} \, f(\varphi_\alpha(\lambda))  \, \, \, (f\in H^m_2, \lambda\in \mathbb{C}_+).
	\end{align}
	Let $c=(\alpha+\overline{\alpha})/2$ and $d=(\alpha-\overline{\alpha})/2i $ so that $\alpha=c+id$ with $ c\in \mathbb{R}$ and $d>0$.	It is readily checked  that $V_\alpha $ is a unitary operator on $H_2^m$ and 
	\begin{align*}
		\left(	V_\alpha^\ast f \right) (\lambda)=\dfrac{1}{\sqrt{d}} \, f\left(\dfrac{\lambda-c}{d}\right) \,\,\,\, (f\in H^m_2, \lambda \in \mathbb{C}_+).
	\end{align*}
	Furthermore, $V_\alpha$ maps $\mathcal{H}(\Theta)$ onto $\mathcal{H}(\widetilde{\Theta})$. By applying Lemma \ref{lem: basic op expression}, we get
	\begin{align}\label{eq: unitary equivalence}
		& (V_\alpha A_\alpha f)(\lambda) \nonumber \\
		&= \sqrt{d}\left \{\left( \dfrac{d \lambda +c-\alpha}{d \lambda+c -\overline{\alpha}}\right) f(d\lambda+c)+\left( \dfrac{\alpha-\overline{\alpha}}{d\lambda+c-\overline{\alpha}}\right) \Theta(d\lambda+c) (\Theta^\# f)(\overline{\alpha}) \right \} \nonumber \\
		&=  \sqrt{d}\left \{ \left(\dfrac{\lambda-i}{\lambda+i}\right) f(d\lambda+c)+\dfrac{2i}{\lambda+i} \Theta(d\lambda+c) (\Theta^\# f) (\overline{\alpha}) \right \}
	\end{align}
	For simplicity, let us write $g(\lambda)=(V_\alpha f)(\lambda)= \sqrt{d}f(d\lambda+c)$. Then 
	\begin{align*}	
		\sqrt{d}(\Theta^\# f) (\overline{\alpha})=\sqrt{d}(\Theta^\# f)(\varphi_\alpha(-i))=(\widetilde{\Theta}^\# g) (-i).
	\end{align*}
	Thus, \eqref{eq: unitary equivalence} becomes
	\begin{align*}
		(V_\alpha A_\alpha f)(\lambda)&= \left(\dfrac{\lambda-i}{\lambda+i}\right)g(\lambda)+\left(\dfrac{2i}{\lambda+i}\right) \widetilde{\Theta}(\lambda) (\widetilde{\Theta}^\# g)(-i) \\
		&= (A_i g) (\lambda) \\
		&= (A_iV_\alpha f)(\lambda).
	\end{align*}
	
	\bigskip
	
	\noindent
	{\bf 3.} \textit{Verification of $(3)$$:$}
	The analysis for $\Omega_+=\mathbb{C}_R$ is similar to the analysis for $\Omega_+=\mathbb{C}_+$. Here $V_\alpha: H^m_2(\mathbb{C}_R)\rightarrow H^m_2(\mathbb{C}_R)$ is defined by the formula
	\begin{align}
		(V_\alpha f)(\lambda)= \sqrt{(\alpha+\overline{\alpha})/2} \, f(\varphi_\alpha(\lambda)) \quad (f\in H^m_2, \lambda \in \mathbb{C}_R),
	\end{align}
	and $\varphi_\alpha(\lambda)= c \lambda +i d$.

	\bigskip
	
	\noindent
	{\bf 4.} \textit{Verification of $(4)$$:$}
	Let $\Theta$ be inner in $\mathbb{D}$ and let	$\Theta_0(\lambda )=\Theta\left(\dfrac{\lambda  -i}{\lambda +i}\right)$ for all $\lambda  \in \mathbb{C}_+$. Then $\Theta_0$ is inner in $\mathbb{C}_+$. Consequently, $\Theta^\#(\omega)=\Theta(1/\overline{\omega})^\ast$ for $\omega \in \mathbb{D}\setminus \{0\}$ and $\Theta_0^\#(\omega)=\Theta_0(\overline{\omega})^\ast$ for $\omega\in \mathbb{C}_+$.
	Define the map $V: H_2^m(\mathbb{D})\rightarrow H^m_2(\mathbb{C}_+)$ by the formula
	\begin{align}
		(Vf)(\lambda )=\dfrac{1}{\sqrt{\pi}(\lambda +i)} f\left(\dfrac{\lambda -i}{\lambda +i}\right) \, \, \, \, \, \, (f\in H_2^m(\mathbb{D}), \lambda  \in \mathbb{C}_+ ).
	\end{align}
	Then it can be checked that $V$ is a unitary operator that maps $\mathcal{H}(\Theta)$ onto $\mathcal{H}(\Theta_0)$.
	Next, to verify the fomula
	\begin{align}\label{Eq: unitary equivalence}
		A_i V f=VA_0f \quad (f\in \mathcal{H}(\Theta)),
	\end{align}	
	let $f\in \mathcal{H}(\Theta)$. Then, by Lemma \ref{lem: basic op expression} (with $\alpha=i$ and $\Omega_+=\mathbb{C}_+$),
	\begin{align}\label{eq: expression 1}
		&	(A_i V f)(\lambda ) \nonumber \\
		&=b_i(\lambda )(Vf)(\lambda )+\left(\dfrac{2i}{\lambda +i}\right)\Theta_0(\lambda )\left( \Theta_0^\# Vf\right)(-i)  \nonumber \\
		&= 	\dfrac{1}{\sqrt{\pi}(\lambda  +i)} \left( \dfrac{\lambda -i}{\lambda +i}\right)f\left( \dfrac{\lambda -i}{\lambda +i}\right)+\left(\dfrac{2i}{\lambda +i}\right)\Theta\left(\dfrac{\lambda  -i}{\lambda +i}\right)  \left({\Theta_0}^\# Vf\right)(-i),
	\end{align}
	whereas, by another application of Lemma \ref{lem: basic op expression}, for $\alpha=0$, $\Omega_+=\mathbb{D}$ and $u= \lim\limits_{\beta\rightarrow 0} \dfrac{ (\Theta^\# f)(1/\overline{\beta})}{\overline{\beta}}$, 
	\begin{align}\label{eq: expression 2}
		(VA_0f)(\lambda )=\dfrac{1}{\sqrt{\pi}(\lambda  +i)} \left( \dfrac{\lambda -i}{\lambda +i}\right) f\left( \dfrac{\lambda -i}{\lambda +i}\right)- \dfrac{1}{\sqrt{\pi}(\lambda  +i)} \Theta\left(\dfrac{\lambda  -i}{\lambda +i}\right) u.
	\end{align}
	Therefore, it remains to show that
	\begin{align}\label{Eq: final equation}
		2i (\Theta_0 ^\# V f)(-i)= -\dfrac{1}{\sqrt{\pi}} u.
	\end{align} 
	Let $\dfrac{1}{\overline{\beta}}=\dfrac{\omega-i}{\omega+i}$ so that $\omega= i\left(\dfrac{\overline{\beta}+1}{\overline{\beta}-1}\right)$. Then $\omega\rightarrow -i$  if and only if $\beta\rightarrow 0 $ and 
	\begin{align*}
		u&= \lim\limits_{\beta\rightarrow 0} \dfrac{ \Theta^\#(1/\overline{\beta}) f(1/\overline{\beta})}{\overline{\beta}} \\
		&= \lim\limits_{\omega\rightarrow -i} (\omega-i) \dfrac{ \Theta^\# \left( \dfrac{\omega-i}{\omega+i}\right) f\left(\dfrac{\omega-i}{\omega+i}\right) }{\omega+i} \\
		&= (-2i) \lim\limits_{\omega\rightarrow -i}  \dfrac{ \Theta_0^\#(\omega) f\left(\dfrac{\omega-i}{\omega+i}\right) }{\omega+i} \\
		&=(-2i)\sqrt{\pi} \lim\limits_{\omega\rightarrow -i}  \Theta_0^\#(\omega) (Vf)(\omega)  \\
		&= -2i\sqrt{\pi} \left( \Theta_0^\# Vf\right)(-i),
	\end{align*}
	which verifies \eqref{Eq: final equation}. Hence the claim in \eqref{Eq: unitary equivalence} follows.

	\bigskip
	
	\noindent
	{\bf 5.} \textit{Verification of $(5)$$:$} The proof is similar to the proof of (4), but now the unitary operator $V: H^m_2(\mathbb{D})\rightarrow H^m_2(\mathbb{C}_R)$ is given by the formula \begin{align}
		(Vf)(\lambda )=\dfrac{1}{\sqrt{\pi}(\lambda +1)} f\left(\dfrac{\lambda -1}{\lambda +1}\right) \, \, \, \, \, \, (f\in H_2^m(\mathbb{D}), \lambda  \in \mathbb{C}_R ).
	\end{align} 
	
\end{proof}
\begin{remark}
	We thank the reviewer for calling our attention to the interesting paper \cite{Cima} and posing queries that pushed us to write this section.
\end{remark}

\section{The de Branges spaces ${\mathcal{B}}({\mathfrak{E}})$}
\label{sec:bofe0}
We begin with a quick introduction to de Branges spaces; for additional discussion see e.g., Sections 3.21 and 5.10 of \cite{Arov-Dym}. 

Let $H_\infty^{p\times q}(\Omega_+)$ (resp., $H_\infty^{p\times q}(\Omega_-)$) denote the set of $p\times q$ matrix valued functions (mvf's for short) with entries that are holomorphic and bounded in $\Omega_+$ (resp., $\Omega_-$) and let $\Pi^{p\times q}$ denote the set of $p\times q$ matrix valued functions that are meromorphic in $\mathbb{C}\setminus\Omega_0$ such that:
\begin{enumerate}
	\item[\rm(1)] $f=g_+/h_+$ in $\Omega_+$ with $g_+\in H_\infty^{p \times q}(\Omega_+)$ and a nonzero $h_+\in H_\infty^{1\times 1}(\Omega_+)$.
	\item[\rm(2)] $f=g_-/h_-$ in $\Omega_-$ with $g_-\in H_\infty^{p \times q}(\Omega_-)$ and a nonzero $h_-\in H_\infty^{1\times 1}(\Omega_-)$. 
	\item[\rm(3)] The nontangential limits  $(g_+/h_+)(\mu )$ and $(g_-/h_-)(\mu )$  at the boundary are equal at almost all points $\mu\in\Omega_0$.
\end{enumerate}

An $m\times 2m$ mvf 
$\mathfrak{E}(\lambda)=\begin{bmatrix}E_-(\lambda)&E_+(\lambda)\end{bmatrix}$
with $m \times m$ blocks $E_\pm(\lambda)$ that are meromorphic in $\Omega_+\cup\Omega_-$ will be called a {\bf de Branges matrix} with respect to $\Omega_+$ 
if 
\begin{equation}
	\label{eq:mar8a22}
	\mathfrak{E}\in\Pi^{m\times 2m},\quad \det\,E_+(\lambda)\not\equiv 0\ \textrm{in}\  \Omega_+\quad\textrm{and}\quad  E_+^{-1}E_-\ 
	\textrm{is an $m \times m$ inner mvf in $\Omega_+$}.
\end{equation}

\begin{remark}
	\label{rem:mar8a22}
	The definitions simplify if $\mathfrak{E}(\lambda)$ is restricted to be entire. Then only the last two constraints in \eqref{eq:mar8a22} are needed; see e.g., \cite{Arov-Dym 3}. 
\end{remark}
If $\mathfrak{E}$ is a de Branges matrix, let
\[
{\mathcal{B}}({\mathfrak{E}})=\{f\in\Pi^{m\times 1}:\,E_+^{-1}f\in H_2^m(\Omega_+)\ominus E_+^{-1}E_-H_2^m(\Omega_+)\}
\]
and set 
\begin{equation}
	\label{eq:feb17a22}
	\langle f,g\rangle_{\mathcal{B}({\mathfrak{E}})}=\langle E_+^{-1}f,E_+^{-1}g\rangle_{st}, \quad\textrm{the inner product defined by} \ \eqref{eq:stip}.
\end{equation}
(We have added the subscript st to the inner product \eqref{eq:stip} because in the next several lines there will be two inner products  in play.) 
Then the space ${\mathcal{B}}(\mathfrak{E})$ equipped with the inner product \eqref{eq:feb17a22}  
is a RKHS with RK
\[
K_\omega^{\mathfrak{E}}(\lambda)=\frac{E_+(\lambda)E_+(\omega)^*-E_-(\lambda)E_-(\omega)^*}{\rho_\omega(\lambda)}.
\]

\begin{definition}
	For each de Branges matrix $\mathfrak{E}=\begin{bmatrix}E_-&E_+\end{bmatrix}$ with $m \times m$ blocks $E_\pm$, let 
	$L_2^m(\Delta_{\mathfrak{E}})$ denote the set of $f$ for which the integral $\langle f,f\rangle_{{\mathcal{B}}({\mathfrak{E}})}$ defined by \eqref{eq:feb17a22} 
	is finite and let 
	$\Pi_{\mathfrak{E}}$ denote the orthogonal projection from $L_2^m(\Delta_{\mathfrak{E}})$ onto ${\mathcal{B}}({\mathfrak{E}})$. 
\end{definition}

The extension of Theorem \ref{thm: main} to de Branges spaces rests on the following theorem which is taken from \cite{Dym}. For the convenience of the reader, we shall sketch the proof for one choice of $\Omega_+$.

\begin{theorem}
	\label{thm:feb22a22}
	If $\mathfrak{E}=\begin{bmatrix}E_-&E_+\end{bmatrix}$ is an $m\times 2m$  de Branges matrix over $\Omega_+$, then  
	\begin{equation}
		\label{eq:feb22b22}
		\Vert \Pi_{\mathfrak{E}}(1/b_\alpha)f\Vert_{{\mathcal{B}}(\mathfrak{E})}^2=\Vert f\Vert_{{\mathcal{B}}({\mathfrak{E}})}^2-\rho_\alpha(\alpha)(E_+^{-1}f)(\alpha)^*(E_+^{-1}f)(\alpha) 
	\end{equation}
	for $f\in{\mathcal{B}}(\mathfrak{E})$ and $\alpha\in\Omega_+$, for all three 
	choices of $\Omega_+$.
\end{theorem}

\begin{proof} We focus on $\Omega_+=\mathbb{C}_R$. The first step in the proof is to show that 
	\begin{equation}
		\label{eq:feb22a22}
		(I-\Pi_{\mathfrak{E}})(1/b_\alpha)f=
		(\alpha+\overline{\alpha})E_+\,\frac{\displaystyle (E_+^{-1}f)(\alpha)}{\displaystyle \delta_\alpha} \quad \text{ for } \alpha\in \mathbb{C}_R.
	\end{equation}
	If $\alpha\in\mathbb{C}_R$, $\delta_\alpha(\lambda)=\lambda-\alpha$, $b_\alpha(\lambda)=(\lambda-\alpha)/(\lambda+\overline{\alpha})$  and 
	$f\in{\mathcal{B}}({\mathfrak{E})}$, then
	$$
	\frac{1}{b_\alpha}=1+\frac{\alpha+\overline{\alpha}}{\delta_\alpha}\implies \Pi_{\mathfrak{E}}(1/b_\alpha)f=f+(\alpha+\overline{\alpha})\Pi_{\mathfrak{E}}\frac{f}{\delta_\alpha}.
	$$
	Thus, if $u\in\mathbb{C}^m$ and $\omega\in\mathbb{C}_R$, then 
	\[
	\begin{split}
		u^*\left(\Pi_{\mathfrak{E}}\frac{f}{\delta_\alpha}\right)(\omega)&=
		\left\langle\Pi_{\mathfrak{E}} \frac{f}{\delta_\alpha}, K_\omega^{\mathfrak{E}} u\right\rangle_{{\mathcal{B}}(\mathfrak{E})}
		=\left\langle \frac{f}{\delta_\alpha}, K_\omega^{\mathfrak{E}} u\right\rangle_{{\mathcal{B}}(\mathfrak{E})}
		=-2\pi  \langle f,\rho_\alpha^{-1}K_\omega^{\mathfrak{E}}u\rangle_{{\mathcal{B}}({\mathfrak{E}})}\\ &=
		-2\pi \left\langle f, \frac{E_+E_+(\omega)^*-E_-E_-(\omega)^*}{\rho_\omega\,\rho_\alpha} u\right\rangle_{{\mathcal{B}}(\mathfrak{E})}\\
		&=-2\pi \left\langle E_+^{-1}f, \frac{E_+(\omega)^*}{\rho_\omega\,\rho_\alpha} u\right\rangle_{st}+
		2\pi \left\langle E_+^{-1}f, (E_+^{-1}E_-)\frac{E_-(\omega)^*}{\rho_\omega\,\rho_\alpha} u\right\rangle_{st}.
	\end{split}
	\]
	The second inner product in the last line is equal to zero, since $E_+^{-1}f$, the first entry in that inner product, is orthogonal to 
	$(E_+^{-1}E_-) H_2^m$, whereas the second entry is in $(E_+^{-1}E_-) H_2^m$. To evaluate the first inner product, we write 
	$$
	\frac{1}{\rho_\omega(\lambda)}-\frac{1}{\rho_\alpha(\lambda)}=\frac{1}{2\pi }\left(\frac{1}{\lambda+\overline{\omega}}-\frac{1}{\lambda+\overline{\alpha}}\right)=(2\pi )\frac{\overline{\alpha}-\overline{\omega}}{\rho_\omega(\lambda)\,\rho_\alpha(\lambda)}
	$$
	and hence that
	\[
	\begin{split}
		-2\pi \left\langle E_+^{-1}f, \frac{E_+(\omega)^*}{\rho_\omega\,\rho_\alpha} u\right\rangle_{st}&=\frac{1}{\omega-\alpha}\,\left\{\left\langle E_+^{-1}f, \frac{E_+(\omega)^*}{\rho_\omega} u\right\rangle_{st}\right.\\ &-\left.
		\left\langle E_+^{-1}f, \frac{E_+(\omega)^*}{\rho_\alpha} u\right\rangle_{st} \right\}\\
		&=\frac{1}{\omega-\alpha}\, u^*E_+(\omega)\{(E_+^{-1}f)(\omega)-(E_+^{-1}f)(\alpha)\}\\
		&=\frac{1}{\omega-\alpha}\,\{ u^*f(\omega)-u^*E_+(\omega)(E_+^{-1}f)(\alpha)\}.
	\end{split}
	\]
	Formula \eqref{eq:feb22a22} drops out by combining these equalities.
	
	To obtain \eqref{eq:feb22b22}, observe that
	\[
	\begin{split}
		\Vert f\Vert_{{\mathcal{B}}(\mathfrak{E})}^2-\Vert \Pi_{\mathfrak{E}}(1/b_\alpha)f\Vert_{{\mathcal{B}}(\mathfrak{E})}^2&=
		\Vert(1/b_\alpha)f\Vert_{{\mathcal{B}}(\mathfrak{E})}^2-\Vert \Pi_{\mathfrak{E}}(1/b_\alpha)f\Vert_{{\mathcal{B}}(\mathfrak{E})}^2\\
		&=\Vert (I- \Pi_{\mathfrak{E}})(1/b_\alpha)f\Vert_{{\mathcal{B}}(\mathfrak{E})}^2\\
		&=\vert \alpha+\overline{\alpha}\vert^2 \Vert E_+\,\frac{(E_+^{-1}f)(\alpha)}{\delta_\alpha}\Vert_{{\mathcal{B}}({\mathfrak{E}})}^2\\
		&=\vert \alpha+\overline{\alpha}\vert^2 (E_+^{-1}f)(\alpha)^*(E_+^{-1}f)(\alpha)\\ &\qquad\qquad \times \,\int_{-\infty}^\infty 
		\vert i\nu-\alpha\vert^{-2}d\nu
	\end{split}
	\]
	agrees with \eqref{eq:feb22b22}, since  the integral is equal to $2\pi /(\alpha+\overline{\alpha})$.
\end{proof}

\begin{theorem}\label{thm:mar1a22}
	Let $\mathfrak{E}=\begin{bmatrix}E_-&E_+\end{bmatrix}$ be a de Branges matrix of size $m\times 2m$  and for each point  $\alpha\in \Omega_+$ at which $\mathfrak{E}$ is holomorphic and $E_+(\alpha)$ is invertible, let
	$$
	B_\alpha=\Pi_{\mathfrak{E}}b_\alpha\vert_{{\mathcal{B}}({\mathfrak{E}})}\quad and\quad {\mathcal{B}}({\mathfrak{E}})_\alpha=\{f\in{\mathcal{B}}({\mathfrak{E}}):\,f(\alpha)=0\}.
	$$
	Then $B_\alpha \in \mathcal{N} \mathcal{A}$ and:
	\begin{enumerate}[\rm(1)]
		\item  $\| B_\alpha\|=1$ if and only if   ${\mathcal{B}}({\mathfrak{E}})_\alpha \neq\{0\}$.
		
		\item 
		If ${\mathcal{B}}({\mathfrak{E}})_\alpha= \{0\}$, $\Theta=E_+^{-1}E_-$ and $s_1\geq s_2\geq \cdots\geq s_m$ are the singular values of $\Theta(\alpha)$, then $$\| B_\alpha\|=\underset{1\leq j\leq m}{\max } \, \{s_j: s_j<1\}.$$
	\end{enumerate}
\end{theorem}

\begin{proof} We first observe that under the given assumptions on $\alpha$, every vector valued function $f\in{\mathcal{B}}({\mathfrak{E}})$ is automatically holomorphic at $\alpha$, since
	$$
	f(\alpha)=E_+(\alpha)\,(E_+^{-1}f)(\alpha)\quad\textrm{and}\quad E_+^{-1}f\in H_2^m(\Omega_+).
	$$ 
	In view of formula \eqref{eq:feb22b22}, the proof of (1) is similar to the proof of (1) in Theorem \eqref{thm: main}.
	
	The verification of (2) rests on the decomposition 
	\[
	{\mathcal{B}}({\mathfrak{E}})={\mathcal{B}}({\mathfrak{E}})_\alpha\oplus \{K_\alpha^{\mathfrak{E}}u:\, u\in\mathbb{C}^m\}.
	\]	
	If ${\mathcal{B}}({\mathfrak{E}})_\alpha=\{0\}$, then $f\in{\mathcal{B}}({\mathfrak{E}})$ if and only if $f=K_\alpha^{\mathfrak{E}}u$ for some $u\in\mathbb{C}^m$. Then, for such $f$,
	\[
	\begin{split}
		\Vert B_\alpha^*K_\alpha^{\mathfrak{E}}u\Vert_{{\mathcal{B}}({\mathfrak{E}})}^2&=\Vert K_\alpha^{\mathfrak{E}}u\Vert_{{\mathcal{B}}({\mathfrak{E}})}^2-
		\rho_\alpha(\alpha)K_\alpha^{\mathfrak{E}}(\alpha)(E_+(\alpha)^*)^{-1}E_+(\alpha)^{-1}K_\alpha^{\mathfrak{E}}(\alpha)u\\
		&=\rho_\alpha(\alpha)^{-1}u^*E_+(\alpha)\{I_m-\Theta(\alpha)\Theta(\alpha)^*-(I_m-\Theta(\alpha)\Theta(\alpha)^*)^2\}E_+(\alpha)^*u\\
		&=\rho_\alpha(\alpha)^{-1}u^*E_+(\alpha)M_\alpha(\alpha)^{1/2}\{I_m-M_\alpha(\alpha)\}M_\alpha(\alpha)^{1/2}E_+(\alpha)^*u,
	\end{split}
	\]
	where $M_\alpha(\alpha)=I_m-\Theta(\alpha)\Theta(\alpha)^*$. Consequently, if $v=M_\alpha(\alpha)^{1/2}E_+(\alpha)^*u\ne 0$, then
	\[
	\frac{\Vert B_\alpha^*K_\alpha^{\mathfrak{E}}u\Vert_{{\mathcal{B}}({\mathfrak{E}})}^2}{\Vert K_\alpha^{\mathfrak{E}}u\Vert_{{\mathcal{B}}({\mathfrak{E}})}^2}=\frac{v^*\{I_m-M_\alpha(\alpha)\}v}{v^*v}=\frac{v^*\Theta(\alpha)\Theta(\alpha)^*v}{v^*v},
	\]
	which, as $E_+(\alpha)$ is invertible, is effectively the same as the right hand side of   \eqref{Eq: important}. Thus, the rest of the proof of Theorem \ref{thm:mar1a22} is exactly the same as the proof of  Theorem \ref{thm: NA on subspace}.
	
	Finally, the proof that $B_\alpha \in \mathcal{N}\mathcal{A}$ is similar to the proof that $A_\alpha \in \mathcal{N}\mathcal{A} $ presented in Theorem \ref{thm: main}. 
\end{proof}

\subsection*{Acknowledgement}
We thank the referee for reading the paper carefully and making useful suggestions.
The research of first named author is supported in part by  the INSPIRE grant of Dr. Srijan Sarkar (Ref: DST/INSPIRE/04/2019/000769), Department
of Science \& Technology (DST), Government of India, the postdoctoral fellowship of IISER Pune, India and the postdoctoral fellowship of Weizmann Institute of Science, Israel.

\subsection*{Competing interests }

The authors declare that there is no conflict of interest.

\subsection*{Data Availability} Not applicable.

\end{document}